\documentclass[pamsfonts, 11pt,a4paper]{amsart}

\usepackage{xcolor}

\newcommand{\F}{\mathbb{F}}

\newtheorem{theorem}{Theorem}[section]
\newtheorem{proposition}{Proposition}[section]

\newtheorem{lemma}{Lemma}[section]
\newtheorem{corollary}{Corollary}[section]

\theoremstyle{definition}

\theoremstyle{remark}
\newtheorem*{remark}{Remark}

\begin{document}
\title{On Quadratic Curves over Finite Fields}
\author{Vagn Lundsgaard Hansen}
\author{Andreas Aabrandt}
\address{Technical University of Denmark}
\maketitle

\begin{abstract}
The geometry of algebraic curves over finite fields is a rich area of research. In \cite{circle}, the authors investigated a particular aspect of the geometry over finite fields of the classical unit circle, namely how the number of solutions of the circle equation depends on the characteristic $p$ and the degree  $n\geq 1$ of the finite field $\F_{p^n}$. In this paper, we make a similar study of the geometry over finite fields of the quadratic curves defined by the quadratic equations in two variables for the classical conic sections. In particular the quadratic equation with mixed term is interesting, and our results display a rich variety of possibilities for the number of solutions to this equation over a finite field.
\end{abstract}

\subjclass{{\small \textbf{Subject class:} 11G20, 11D45, 11D09, 11A07, 14G15}}

\keywords{{\small \textbf{Keywords:} Diophantine geometry, curves over finite fields, counting solutions to quadratic equations}}

\section{Introduction}

The geometry of algebraic curves over finite fields is a fascinating subject which emerged as an important research area in works of Abel and Galois in the 1820s and gained momentum in an inspired survey paper on the number of solutions of equations in finite fields published 1949 by Andr\'e Weil \cite{weil49}. 

In \cite{circle}, we initiated a study of how the number of solutions over a finite field $\F_{p^n}$ of the polynomial equation for an algebraic curve depends on the characteristic $p$ and the degree  $n\geq 1$ of the field. Our study of the circle equation
$$x^2+y^2=1,$$
revealed that surprising phenomena can happen. 

In this paper, we make a similar study of the geometry over the finite fields 
$\F_{p^n}$ of the quadratic curves defined by the quadratic equations in two variables for the classical conic sections, cf. (\cite{Hansen98}, Section 2.6). For simplicity we assume that all constants in the equations equals $1$.

We have already examined the {\em elliptic equation} (circle equation) in \cite{circle}. Here we shall examine the following types of quadratic equations:
\vspace{2mm}

\quad The {\em hyperbolic equation} $$x^2-y^2=1.$$

\quad The {\em parabolic equation} $$y=x^2.$$

\quad The {\em quadratic equation with mixed term} $$x^2+xy+y^2=1.$$

In particular the quadratic equation with mixed term is interesting. Our results display a rich variety of possibilities for the number of solutions to the equation depending on the characteristic $p$ and the degree $n\geq 1$ of the field $\F_{p^n}$. The prime $p=2$ is especially intriguing.

\section{Solutions to the hyperbolic equation}

\begin{theorem}
Over the finite field $\;\F_{2^n}$ corresponding to the prime $p=2$ and the integer $n\geq 1$, the hyperbolic equation
$$x^2-y^2 = 1$$
has exactly $2^n$ solutions of ordered pairs $(x,y)$ of elements in $\;\F_{2^n}$.
\end{theorem}

\begin{proof}
For $p=2$, the equation $x^2-y^2 = 1$ coincides with the circle equation $x^2+y^2 = 1$. Hence the result follows from (\cite{circle}, Corollary 2.1).
\end{proof}

\begin{theorem}\label{theorem:hypeq}
Over the finite field $\F_{p^n}$ corresponding to an odd prime $p\geq 3$ and the integer $n\geq 1$, the hyperbolic equation 
$$x^2-y^2 = 1$$
has exactly $p^n-1$ solutions of ordered pairs $(x,y)$ of elements in $\;\F_{p^n}$.
\end{theorem}

\begin{proof}
The equation $x^2-y^2 = 1$ is equivalent to the equation
$$(x+y)(x-y) = 1.$$

From this follows that the ordered pair $(x,y)$ of elements $x,y\in \F_{p^n}$ is a solution to $x^2-y^2 = 1$ over 
$\F_{p^n}$ if and only if
$$x+y\neq 0 \quad \text{and} \quad x-y = (x+y)^{-1} .$$

Now put $x+y=u$ for $u\neq 0$. Then the ordered pair $(x,y)$ is a solution to $x^2-y^2 = 1$ over $\F_{p^n}$ 
if and only if $x-y=u^{-1}$.

Let $\F_{p^n}^*$ denote the multiplicative group of non-zero elements in $\F_{p^n}$.

From the above we can then conclude that the ordered pair $(x,y)$ is a solution to $x^2-y^2 = 1$ over $\F_{p^n}$ if and only if
$$x=2^{-1}(u+u^{-1}) \quad \text{and} \quad y=2^{-1}(u-u^{-1})$$
for an arbitrary element $u\in\F_{p^n}^*$. 

Since the order of $\F_{p^n}^*$ is $p^n-1$, the theorem follows. 
\end{proof}

In Table 1, we display for each of the primes $p=2,3,5,7,11$, the set of all ordered pairs $(x,y)$ of elements in the prime field $\F_p$ that constitutes the set of solutions and the number $N_p$ of solutions to the hyperbolic equation over $\F_p$.
\bgroup
\def\arraystretch{1.8}%
\begin{center}
\begin{table}
\caption{Solutions to $x^2-y^2 = 1$ for $p=2,3,5,7,11$.\label{table:Table1}}
    \begin{tabular}{ | l | p{8cm}| l |}
    \hline
    $p$ & Solutions to $x^2-y^2=1$ & $N_p$ \\ \hline
	$2$ & $(1,0),(0,1)$ & 2
	\\ \hline
    $3$ & $(1,0),(2,0)$ & 2
    \\ \hline
    $5$ & $(1,0),(0,2),(0,3),(4,0)$ & 4
    \\ \hline
    $7$ & $(1,0),(3,6),(4,6),(3,1),(4,1),(6,0)$ & 6
    \\ \hline
    $11$ & \parbox{5cm}{$(1,0),(4,9),(9,5)$,$(9,6),(7,9),$\\$(4,2),(2,5)$,$(2,6),(7,2),(10,0)$} & 10
    \\ \hline
   	
    \end{tabular}
\end{table}
\end{center}
\egroup

\section{Solutions to the parabolic equation}

\begin{proposition}
Over the finite field $\;\F_{p^n}$ corresponding to the prime $p\geq 2$ and the integer 
$n\geq 1$, the parabolic equation 
$$y=x^2$$
has exactly $p^n$ solutions of ordered pairs $(x,y)$ of elements in $\;\F_{p^n}$.
\end{proposition}

\begin{proof}
This is trivially true since for each $x\in \F_{p^n}$ there is a unique element 
$y\in \F_{p^n}$ such that $y=x^2$. Since the order of $\;\F_{p^n}$ is $p^n$, there are therefore exactly $p^n$ ordered pairs $(x,y)$ of elements in $\;\F_{p^n}$ solving the parabolic equation.
\end{proof}

\section{The quadratic equation with mixed term for odd primes}

In this section we examine the quadratic equation
$$x^2+xy+y^2=1$$
over the finite fields $\;\F_{p^n}$ of characteristic $p\geq 3$ and dimension 
$n\geq 1$.

For any prime $p\geq 3$, the equation $x^2+xy+y^2=1$ can be rewritten as follows by completion of the square
$$(x+2^{-1}y)^2+(1-(2^{-1})^2)y^2=1.$$

Note that the inverse to $2\in \F_{p^n}$ for $p\geq 3$ is given by 
$2^{-1}=(p+1)/2$.
Hence we can rewrite the equation $x^2+xy+y^2=1$ as
$$\big(x+\frac{p+1}{2}y\big)^2+\big(1-\big(\frac{p+1}{2}\big)^2\big)y^2=1,$$
which can be simplified to
$$\big(x+\frac{p+1}{2}\;y\big)^2+\frac{(3+p)(1-p)}{4}\;y^2=1 .$$

Finally we can then rewrite the equation $x^2+xy+y^2=1$ as 
$$z^2+\frac{(3+p)(1-p)}{4}\;y^2=1 \quad \text{with} \quad z=x+\frac{p+1}{2}\;y.$$ 

Define the element $a(p)\in \F_p$  for $p\geq 3$ by
$$a(p)=\frac{(3+p)(1-p)}{4}.$$

The rewriting of the equation $x^2+xy+y^2=1$ can then be formulated as
$$z^2+a(p)y^2=1 \quad \text{with} \quad z=x+\frac{p+1}{2}\;y.$$ 

The element $a(p)\in \F_p$ for $p\geq 3$ plays a prominent role in determining the structure of the set of solutions to the equation $x^2+xy+y^2=1$ over the finite fields $\;\F_{p^n}$ for $n\geq 1$. 

\begin{lemma}\label{a(p)lemma}
For $p\geq 3$ and $n\geq 1$, the element $a(p)\in \F_p$ has the properties.
\begin{enumerate}
\item $a(p)=0$ in $\F_p$ if and only if $p=3$.
\vspace{2mm}
\item $a(p)=-1$ in $\F_p$ if and only if $p=7$.
\item $a(p)$ is a square in $\F_{p^n}^*$ if and only if $a(p)^{\frac{p^n-1}{2}}=1$.
\end{enumerate}
\end{lemma}

\begin{proof}
(1) follows by observing that $a(p)=0$ if and only if 
$$(3+p)(1-p)\equiv 0 \pmod p ,$$
or equivalently, if and only if
$$3\equiv 0 \pmod p ,$$
which happens only for $p=3$.

(2) follows by observing that $a(p)=-1$ if and only if 
$$(3+p)(1-p)\equiv -4 \pmod p ,$$
or equivalently, if and only if
$$3\equiv -4 \pmod p ,$$
which happens only for $p=7$.

(3) follows from (\cite{powers}, Theorem 2).
\end{proof}
 
Using the results from Lemma \ref{a(p)lemma} about the element $a(p)\in \F_p$, we can now determine the number of solutions to the quadratic equation with mixed term for the primes $p=3$ and $p=7$ and all $n\geq 1$.
 
\begin{theorem}\label{theorem_three}
For $p=3$ and an arbitrary integer $n\geq 1$, the equation 
$$x^2+xy+y^2=1$$ 
has exactly $2\cdot 3^n$ solutions of ordered pairs $(x,y)$ of elements in 
$\;\F_{3^n}$.
\end{theorem}

\begin{proof}
For $p=3$, the coefficient $a(3)=0$, and hence finding solutions to the equation $x^2+xy+y^2=1$ reduces to finding solutions to the equations
$$z^2=1 \quad \text{with} \quad z=x+2y.$$

Since $z=\pm\;1$, we only have to find the number of solutions to the two equations 
$x+2y=\pm\; 1$ in $\;\F_{3^n}$. 

For any choice of $y\in\F_{3^n}$, there exists for each of the two equations, a unique $x\in\F_{3^n}$ such that the equation is satisfied. The field $\F_{3^n}$ has $3^n$ elements, and hence there are exactly $2\cdot 3^n$ solutions of ordered pairs $(x,y)$ of elements in the finite field $\;\F_{3^n}$ satisfying $x^2+xy+y^2=1$.
\end{proof}

\begin{theorem}
For $p=7$ and an arbitrary integer $n\geq 1$, the equation 
$$x^2+xy+y^2=1$$ 
has exactly $7^n-1$ solutions of ordered pairs $(x,y)$ of elements in $\;\F_{7^n}$.
\end{theorem}

\begin{proof}
For $p=7$, the coefficient $a(7)=-1$, and hence finding solutions to the equation $x^2+xy+y^2=1$ reduces to finding solutions to the equations
$$z^2-y^2=1 \quad \text{with} \quad z=x+4y.$$

By Theorem \ref{theorem:hypeq}, the equation $z^2-y^2=1$ has exactly $7^n-1$ solutions of ordered pairs $(z,y)$ of elements in $\;\F_{7^n}$. For any choice of $z,y\in\F_{7^n}$, there exists a unique $x\in\F_{7^n}$ such that the equation $z=x+4y$ is satisfied. It follows that the equation $x^2+xy+y^2=1$ has exactly $7^n-1$ solutions of ordered pairs $(x,y)$ of elements in the finite field $\;\F_{7^n}$.
\end{proof}

To obtain results on the number of solutions to the quadratic equation with mixed term for odd primes $p\neq 3,7$ is more subtle. As we shall see we can make progress if the element $a(p)\in \F_p$ is a square in $\F_{p^n}^*$. 

\begin{remark}
By direct computations it can be shown that $a(p)$ is a square in $\F_p^*$ for the primes $p=11,13,23,37,47$ and hence in $\F_{p^n}^*$ for all $n\geq 1$. In fact, these primes are the first five primes known to have this property.
\end{remark}

A family of cases where the element $a(p)\in \F_p$ is a square in $\F_{p^n}^*$ occurs as a corollary to the following general theorem. 

\begin{theorem}\label{theorem_all_squares}
For an odd prime $p\geq 3$ and $n\geq 2$ an even integer, it holds that every element in $\F_{p}^*$ is a square in $\F_{p^n}^*$.
\end{theorem}
\begin{proof}
The finite field $\F_{p^n}$ is uniquely determined up to isomorphism as the splitting field for the polynomial 
$f_n(x)=x^{p^n}-x$ over the prime field $\F_p$. 

For $n\geq 2$ an even integer and any prime $p\geq 2$, we have
$$p^n-1= (p^2-1)q(p),$$
and
$$x^{p^n}-x=(x^{p^2}-x)g(x),$$
where 
$$q(p)=1+p^2+p^4+p^6+\dots+p^{n-2},$$
and 
$$g(x)=\sum_{k=1}^{q(p)-1}{x^{(p^n-1)-k(p^2-1)}} + 1.$$

The finite field $\F_{p^2}$ is the splitting field for the polynomial $f_2(x)=x^{p^2}-x$ over $\F_p$ and since $f_2(x)$ is a factor in $f_n(x)$, we can therefore identify $\F_{p^2}$ with a subfield of $\F_{p^n}$. 
Therefore we only need to prove the theorem for the case $n=2$.

For an odd prime $p\geq 3$ and an arbitrary element $a\in\F_p^*$ we have the computations
$$a^{(p^2-1)/2}=a^{(p-1)(p+1)/2}=(a^{p-1})^{(p+1)/2}=1,$$
since $a^{p-1}=1$ by Fermat's little theorem, cf. \cite{Davenport}. 

Using the Generalized Euler's Criterion (\cite{powers}, Theorem 2) this shows that $a\in\F_p^*$ is a square in $\F_{p^2}$, and hence in $\F_{p^n}$, for every odd prime 
$p\geq 3$ and every even integer $n\geq 2$.
\end{proof}

\begin{corollary}\label{corollary_a(p)square}
For all primes $p\geq 5$ and $n\geq 2$ an even integer, the element $a(p)\in \F_p$ is a square in $\F_{p^n}^*$.
\end{corollary}
\begin{proof}
By Lemma \ref{a(p)lemma} the element $a(p)\in \F_p^*$ for $p\neq 3$. Then it follows immediately
by Theorem \ref{theorem_all_squares} that $a(p)$ is a square in $\F_{p^n}^*$ for all primes $p\geq 5$ and 
$n\geq 2$ an even integer.
\end{proof}

Corollary \ref{corollary_a(p)square} gives weight to the following.

\begin{theorem}
Let $n\geq 1$ be an arbitrary integer, and let $p\geq 3$ be a prime for which $a(p)$ is a square in $\F_{p^n}^*$. Then the number of solutions to the equation 
$$x^2+xy+y^2=1$$ 
over the finite field $\;\F_{p^n}$ is given by the formula
$$N_{p^n}=p^n-\sin\big(p^n\frac{\pi}{2}\big).$$
\end{theorem}

\begin{proof}
We have reduced the problem of finding the number of solutions to the equation 
$x^2+xy+y^2=1$ over the finite field $\;\F_{p^n}$ to finding the number of solutions to the equations 
$$z^2+a(p)y^2=1 \quad \text{with} \quad z=x+\frac{p+1}{2}\;y,$$ 
where $a(p)=b^2$ for an element $b\in\F_{p^n}^*$.

Now put $u=b\;y$. Then the problem is reduced to finding the number of solutions to the equations
$$z^2+u^2=1 \quad \text{with} \quad z=x+\frac{p+1}{2}\;b^{-1}\;u .$$ 

From (\cite{circle}, Theorem 4.1)we know that the number of ordered pairs $(z,u)$ of elements in $\;\F_{p^n}$ solving the circle equation $z^2+u^2=1$ is given by
$$N_{p^n}=p^n-\sin\big(p^n\frac{\pi}{2}\big).$$
For any choice  $z,u\in\F_{p^n}$, there exists a unique $x\in\F_{p^n}$ such that the equation 
$$z=x+\frac{p+1}{2}\;b^{-1}\;u$$ 
is satisfied. It follows that the equation $x^2+xy+y^2=1$ has exactly $N_{p^n}$
solutions of ordered pairs $(x,y)$ of elements in the finite field $\;\F_{p^n}$.
\end{proof}

\section{The quadratic equation with mixed term in characteristic 2}
In this section we examine the quadratic equation
$$x^2+xy+y^2=1$$
over the finite fields $\;\F_{2^n}$ of characteristic $2$ and degree $n\geq 1$.

We begin by making a general study of the equation 
$$x^2+xy+y^2=c$$ 
for an arbitrary $c\in \F_{2^n}$. 

\begin{theorem}\label{SolutionEquality}
For $c\neq 0$, all the equations $$x^2+xy+y^2=c$$ have the same number of solutions of ordered pairs $(x,y)$ of elements in the finite field $\;\F_{2^n}$.
\end{theorem}
\begin{proof}
For any element $c=d^2\in \F_{2^n}^*$, multiplication by $c=d^2$ defines an isomorphism of $\F_{2^n}^*$ mapping $1\in\F_{2^n}^*$ into $c\in \F_{2^n}^*$. By the similar isomorphism defined by multiplication by $d$, the set of solutions to the equation $x^2+xy+y^2=1$ is mapped bijectively onto the set of solutions to the 
equation $x^2+xy+y^2=c$. Hence the equations for $c=d^2\in \F_{2^n}^*$ all have the same number of solutions. Since the squaring homomorphism $x^2:\F_{2^n}\rightarrow\F_{2^n}$ is an isomorphism (\cite{circle}, proof of Theorem 2.1), every element $c\in \F_{2^n}^*$ is in fact a square $c=d^2$. This proves that all the equations $x^2+xy+y^2=c$ for $c\in \F_{2^n}^*$ have the same number of solutions.
\end{proof}

Making use of Theorem \ref{SolutionEquality}, we can determine the exact number of solutions to the equation for all $c\in\F_{2^n}$, if we can determine it for $c=0$.

Over $\F_{2^n}$ the equation $$x^2+xy+y^2=0$$ is equivalent to the 
equation $$(x+y)^2=xy.$$  

By introducing the extra variable $u\in\F_{2^n}$, we can rewrite this equation as the system of equations
\[
\begin{split}
x+y  &= u,\\
xy &= u^2.
\end{split}
\]

(a) If $u=0$, we first get $y=x$ and then $x^2=0$. Since 
$x^2:\F_{2^n}\rightarrow\F_{2^n}$ is an isomorphism, it follows that $x=y=0$, giving the solution $(x,y)=(0,0)$.
\vspace{1mm}

(b) If $u\neq 0$, we can rewrite the system of equations to be solved to  
\[
\begin{split}
u^{-1}x+u^{-1}y&=1,\\
u^{-1}x\cdot u^{-1}y&=1.
\end{split}
\]

Put $\bar x = u^{-1}x$ and $\bar y = u^{-1}y$. Then the system takes the form
\[
\begin{split}
\bar x + \bar y &= 1,\\
\bar x \bar y &=1.
\end{split}
\]

By a final rewriting, we first get
\[
\bar y = \bar x^{-1},
\]
and then 
\[
\bar x + \bar x^{-1}=1,
\]
which in $\F_{2^n}$ is equivalent to the equation in one variable
$$\bar x^2 + \bar x + 1 = 0.$$ 

The number of solutions to this equation depends on the parity of the degree $n\geq 2$ of the field. 

\begin{lemma}\label{Lemma-char2}
Consider the equation $$\bar x^2+\bar x +1=0$$ over the finite field $\F_{2^n}$.
\begin{enumerate}
\item For $n\geq 2$ an even number, the equation has two solutions.
\item For $n\geq 3$ an odd number, the equation has no solutions.
\end{enumerate}
\end{lemma}

\begin{proof}
(1) It can easily be checked by direct computation, that the equation has two solutions in $\F_{2^2}$, both lying outside the prime field. This implies that it also has two solutions over any finite field $\F_{2^n}$ of even degree $n\geq 2$, since we know from the proof of Theorem \ref{theorem_all_squares}, that for all even integers $n\geq 2$, the field $\F_{2^2}$ is isomorphic to a subfield of $\F_{2^n}$.

(2) Represent the field $\F_{2^n}$ as the quotient field $\F_2[t]/({\rm Irr}(t))$ of the polynomial ring $\F_2[t]$ modulo an irreducible polynomial ${\rm Irr}(t)$ of degree $n$, cf. \cite{lang}. An arbitrary element in $\F_{2^n}$ then has the form  $\bar x=a_0+a_1t+\dots +a_{n-1}t^{n-1}$ for $a_0, a_1, \dots,a_{n-1}\in \F_2$. Since $n\geq 3$ is an odd number, it follows by consideration of degrees of polynomials that no element $\bar x\in \F_{2^n}$ can solve the equation $\bar x^2+\bar x +1=0$.
\end{proof}
\vspace{2mm}

Collecting facts we get the following result on the number of solutions to the quadratic equation with mixed term in characteristic 2.

\begin{theorem}\label{Theorem-char2}
The quadratic equation
$$x^2+xy+y^2=1$$
has exactly $$2^n +(-1)^{n-1}$$ solutions of ordered pairs $(x,y)$ of elements in the finite field $\;\F_{2^n}$ of characteristic $2$ and degree $n\geq 1$.
\end{theorem}

\begin{proof}
We divide the proof into three cases.

(a) For $n=1$, there are three solutions, namely $(x,y)=(1,0),(1,1),(0,1)$, in accordance with the formula $2^n +(-1)^{n-1}$.
\vspace{.5mm}

(b) For $n\geq 2$ even, there are two solutions to the equation $\bar x^2+\bar x +1=0$ for each $u\in\F_{2^n}^*$ by Lemma \ref{Lemma-char2}. 
Since $u\in\F_{2^n}^*$ can assume $2^n-1$ values, we get in this way $2\cdot (2^n-1)$ non-trivial solutions to the equation $x^2+xy+y^2=0$. 
In addition we also have the trivial solution $(x,y)=(0,0)$, so that altogether there are $2\cdot (2^n-1)+1$ solutions to the equation $x^2+xy+y^2=0$.

All of the $2^n-1$ equations $x^2+xy+y^2=c$ for $c\neq 0$, have the same number of solutions by Theorem \ref{SolutionEquality}. Therefore  
the number of solutions to the equation $x^2+xy+y^2=1$ for $n\geq 2$ even is given by
$$
\frac{2^n\cdot 2^n - [2\cdot (2^n-1)+1]}{2^n -1} = 2^n-1 = 2^n +(-1)^{n-1}.
$$

(c) For $n\geq 3$ odd, it follows by Lemma \ref{Lemma-char2}, that there is only one solution to the equation $x^2+xy+y^2=0$, namely the trivial solution $(x,y)=(0,0)$.

Since again all of the $2^n-1$ equations $x^2+xy+y^2=c$ for $c\neq 0$, have the same number of solutions by Theorem \ref{SolutionEquality}, it follows that 
the number of solutions to the equation $x^2+xy+y^2=1$ for $n\geq 3$ odd is given by
$$
\frac{2^n\cdot 2^n - 1}{2^n -1} = 2^n+1= 2^n +(-1)^{n-1}.
$$
\end{proof}

We finish this section with an application of Theorem \ref{Theorem-char2} to study the number of solutions to the the quadratic equation
$x^2+xy+y^2=1$ over the finite field $\;\F_{2^n}$ as a function of $n\geq 1$. We need the following Lemma on prime powers of $2$.

\begin{lemma}\label{lemma:char2_mod3}
For any prime power $2^n$, $n\geq 1$, it holds that $2^n\equiv 1 \pmod 3$ for $n\geq 2$ even, and $2^n\equiv 2 \pmod 3$ for $n\geq 1$ odd. 
\end{lemma}
\begin{proof}
A prime power $2^m$, $m\geq 1$, can never be divisible by the prime $3$, and hence $2^m\equiv 1,2 \pmod 3$.

If $n\geq 2$ is even, we can write $n=2m$, $m\geq 1$. If $2^m\equiv 1 \pmod 3$, then also $2^n=2^m\cdot 2^m\equiv 1 \pmod 3$. If $2^m\equiv 2 \pmod 3$, it follows likewise that $2^n=2^m\cdot 2^m\equiv 4\equiv 1 \pmod 3$. Hence $2^n\equiv 1 \pmod 3$ for $n\geq 2$ even.

If $n\geq 3$ is odd, we can write  $n=2m+1$, $m\geq 1$. Then it follows that $2^n=2\cdot 2^{2m}\equiv 2\cdot 1=2 \pmod 3$. 
\end{proof}

\begin{theorem}
Over a finite field of characteristic $2$, the number of solutions to the equation 
$$x^2+xy+y^2=1$$
grows in multiples of $3$ as a function of the degree of the extension.
\end{theorem}
\begin{proof}
By Theorem \ref{Theorem-char2}, the number $N_{2^n}$ of solutions to $x^2+xy+y^2=1$ over the finite field $\F_{2^n}$, $n\geq 1$, 
is given by $N_{2^n}=2^n +(-1)^{n-1}$.

Then we have the following computation.
\[
\begin{split}
N_{2^{n+1}}-N_{2^n} &= 2^{n+1}-2^n + (-1)^n - (-1)^{n-1}\\
& = 2^n - 2\cdot (-1)^{n-1} .
\end{split}
\]
Making use of Lemma \ref{lemma:char2_mod3}, it is now easy to prove that $$N_{2^{n+1}}-N_{2^n}\equiv 0 \pmod 3 ,$$ 
and thereby completing the proof of the theorem.
\end{proof}

\bibliographystyle{plain}

\begin{thebibliography}{1}
\bibitem{powers} {Andreas {Aabrandt} and Vagn {Lundsgaard Hansen}}. {A Note on Powers in Finite Fields}. Internat. J. Math. Ed. Sci. Tech. 47(2016), No. 6, 987--991.

\bibitem{circle} {Andreas {Aabrandt} and Vagn {Lundsgaard Hansen}}. {The Circle Equation over Finite Fields}. Quaest. Math. (to appear).

\bibitem{Davenport} Harold {Davenport}. {The Higher Arithmetic. An Introduction to the Theory of Numbers}. {Dover Publications, Inc., New York}, {1983}.

\bibitem{Hansen98} Vagn {Lundsgaard Hansen}. {Shadows of the Circle}. {World Scientific}, Singapore, 1998.

\bibitem{lang} Serge {Lang}. {Algebra}. {Springer}, Reading, Massachusetts, 2005.

\bibitem{weil49} {Andr\'e {Weil}}. {Numbers of solutions of equations in finite fields}. {Bull. Amer. Math. Soc.}, 55:497-508, 1949.



\end{thebibliography}

\end{document}